\documentclass[twoside,12pt]{article}
\usepackage{fullpage}
\usepackage{graphicx}
\usepackage{amsmath,amsthm}
\usepackage{amssymb,amsfonts,amsbsy}
\usepackage{dsfont,eufrak,mathrsfs}
\usepackage[utf8]{inputenc}
\usepackage{color}
\usepackage[T1]{fontenc}
\usepackage{tikz}

\newtheorem{Lema}{Lemma}[section]
\newtheorem{Teo}[Lema]{Theorem}

\newtheorem{Def}[Lema]{Definition}
\newtheorem{Cor}[Lema]{Corollary}
\newtheorem{Rem}[Lema]{Remark}
\newcommand*{\dsty}{\displaystyle}
\newcommand*{\fun}[3]{#1: #2 \rightarrow #3}

\newcommand*{\ColonEqq}{\mathrel{\mathop:}=}
\newcommand*{\Hom}[1]{\mathrm{Homeo}^{+}(#1)}

\newcommand*{\Zedp}{\mathbb{Z}^{+}}

\newcommand{\mcg}{{\rm MCG^*}}

\newcommand*{\Int}[1]{\mathrm{int}(#1)}
\newcommand*{\id}[1]{\mathrm{id}_{#1}}
\newcommand*{\tld}[1]{\widetilde{#1}}

\title{The Alexander method for infinite-type surfaces}
\author{Jes\'{u}s Hern\'{a}ndez Hern\'{a}ndez, Israel Morales, Ferr\'{a}n Valdez}
\date{}

\begin{document}
\maketitle
\begin{abstract}
  We prove that for any infinite-type orientable surface $S$ there exists a collection of essential curves $\Gamma$ in $S$ such that any homeomorphism that preserves the isotopy classes of the elements of $\Gamma$ is isotopic to the identity. The collection $\Gamma$ is countable and has infinite complement in $\mathcal{C}(S)$, the curve complex of $S$. As a consequence we obtain that the natural action of the extended mapping class group of $S$ on $\mathcal{C}(S)$ is faithful.
\end{abstract}
\section{Introduction}

Let $S$ be an orientable surface and $\mcg(S)$ the \emph{extended mapping class group} of $S$, that is the group of \emph{all} homeomorphisms of $S$ modulo isotopy. There is a natural non-trivial action of $\mcg(S)$ on $\mathcal{C}(S)$, the curve complex of $S$. This is 
the abstract simplicial complex whose vertices are the (isotopy classes of) essential curves in S, and whose simplices are multicurves of finite cardinality, see \cite{FarbMar} for a detailed discussion. 
%\begin{equation}
%\Psi:\mcg(S)\to {\rm Aut}(\mathcal{C}(S)).
%\end{equation}
It is well-known that this action is faithful whenever $S$ has finitely generated fundamental group (except for the closed surface of genus 2), see \cite{Ivanov}, \cite{Luo}, \cite{Korkmaz}. The proof of this fact follows from the so-called \emph{Alexander method} (see \cite{FarbMar}, \S 2.3 for details). Roughly speaking, this ``method''  states that the isotopy class of a homeomorphism of $S$ is
often determined by its action on a well-chosen collection of curves and arcs
in S. %Thus, there is a concrete way to determine when a class $f\in\mcg(S)$ acts trivially on $\mathcal{C}(S)$. 

The main purpose of this article is to extend the Alexander method for \emph{all} infinite-type surfaces, \emph{i.e.} when $\pi_1(S)$ \emph{is not} finitely generated. More precisely, we prove:
% MAIN THEOREM
\begin{Teo}
	\label{MdeA}
Let $S$ be an orientable surface of infinite topological type, with possibly non-empty boundary. There exists a collection of essential arcs and simple closed curves $\Gamma$ on $S$ such that any orientation-preserving homeomorphism fixing pointwise the boundary of $S$ that preserves the isotopy classes of the elements of $\Gamma$, is isotopic to the identity.
\end{Teo}

As a matter of fact, we give a recipe to construct the collection $\Gamma$
in \S 
\ref{Sec:AlexanderMethod}, from which it is easy to see that $\Gamma$ is a ``small'' countable collection of curves (\emph{i.e.} $\mathcal{C}^0(S)\setminus\Gamma$ is infinite). Moreover, we prove that there are uncountably many different examples of this kind of collections, see lemma \ref{UncountablyMany}. Recall that any homeomorphism of $S$ that reverses orientation does not act trivially on $\mathcal{C}(S)$, see \cite{McCarthyPapadoupoulos}. As an immediate consequence of Theorem \ref{MdeA} we obtain:
%COROLLARY
\begin{Cor}
	\label{Corollary:FaithfulAction}
Let $S$ be an orientable surface of infinite topological type with empty boundary. Then the natural action of the extended mapping class group of $S$ on the curve complex $\mathcal{C}(S)$ is faithful.
\end{Cor}
%REMARK

It is important to remark that the preceding corollary was already known for infinite type surfaces $S$ for which all ends\footnote{See \S \ref{SECTION:Preliminaries} below for details on the space of ends of a surface.} carry (infinite) genus, see \cite{HernandezValdez}. However, the methods used in that paper to prove that a homeomorphism $h\in{\rm Homeo(S)}$ is isotopic to the identity require that $h$ fixes \emph{all} isotopy classes in $\mathcal{C}(S)$. Therefore the results we present here are both an extension and a refinement of the results obtained in [\emph{Ibid.}]. Moreover, corollary \ref{Corollary:FaithfulAction} will be used in a second paper\footnote{We decided to write a second paper since the methods  used to prove that any simplicial automorphism of $\mathcal{C}(S)$ is geometric are very different from the ones we present here. } to prove that the group of simplicial automorphisms of $\mathcal{C}(S)$ is naturally isomorphic to $\mcg(S)$ for \emph{any} infinite-type surface. This result is used in \cite{AramayonaFounar} to prove that every automorphism of a certain subgroup \footnote{These are called by the authors \emph{asymtotically rigid mapping class groups.}} of $\mcg(S)$ is induced by a homeomorphism of $S$, when $S$ is a genus $g\geq 0$  closed surface from which a Cantor set has been removed.

\textbf{Reader's guide}. In \S \ref{SECTION:Preliminaries}  we  make a short discussion on infinite-type surfaces and the generalities about curves, arcs and continuous deformations of elements in ${\rm Homeo}^+(S)$. The real content of the paper is found in \S \ref{Sec:AlexanderMethod} where we prove Theorem \ref{MdeA}. At the end we added an appendix to address some technical aspects of the proofs we present in section \ref{Sec:AlexanderMethod}.

\textbf{Acknowledgements}. 
The first author was supported during the creation of this article by the UNAM Post-Doctoral Scholarship Program 2016 at the CCM-UNAM.
The second author was generously supported by CONACYT and CCM-UNAM. The third author was generously supported by LAISLA and PAPIIT IN100115.

%%%%%%%%%%%%%%%%%%%%%%%%%%%%%%%%%%%%%%%%%%%%%%%%%
%%%%%%%%%%%%%%%%%%%%%%%%%%%%%%%%%%%%%%%%%%%%%%%%%
%%%%%%%%%%%%%%%%%%%%%%%%%%%%%%%%%%%%%%%%%%%%%%%%%
%%%%%%%%%%%%%%%%%%%%%%%%%%%%%%%%%%%%%%%%%%%%%%%%%
\section{Preliminaries}
\label{SECTION:Preliminaries}

\indent Let $S$ be an orientable surface. We say $S$ is a finite-type surface if its fundamental group is finitely generated; otherwise, we say $S$ is of infinite type. 

An infinite-type surface $S$ is determined, up to homeomorphism, by two invariants: its genus $g(S)\in\mathbb{N}\cup\{\infty\}$ and a pair of nested topological spaces $\rm Ends_\infty(S)\subset Ends(S)$ forming what is known as \emph{the space of ends} of $S$. 
Roughly speaking, the space $\rm Ends(S)$ codifies the different ways in which a surface tends to infinity and $\rm Ends_\infty(S)$ is the subspace formed by all those ends  that carry (infinite) genus. The spaces $\rm Ends(S)_\infty\subset Ends(S)$ are homemorphic to a pair of nested closed subsets of the standard triadic Cantor set. Conversely, every infinite-type topological surface can be constructed from such a pair $X_\infty\subset X$ as follows: think of $X_\infty\subset X$ as living in the unit interval used to construct the standard triadic Cantor set. Think of the unit interval as a subset of the sphere $\mathbb{S}^2$ and imagine that this sphere lives in $\mathbb{R}^3$. Push points in $X$ infinitely away from the origin. The resulting surface $S'$ has a space of ends homeomorphic to $X$. Now, for each $x\in X_\infty\subset X$ add a divergent sequence of handles to $S'$ that accumulates to $x$. This produces a surface $S$ whose space of ends is homeomorphic to $X_\infty\subset X$.

We refer the reader to the work of Ian Richards \cite{Richards} for a more detailed discussion on the topological classification of infinite-type surfaces. We do  not discuss these results further on since we do not need them for the proof of our results.

If $S$ is a finite-type surface of genus $g$, $n$ punctures and $b$ boundary components, we denote by $\kappa(S) = 3g -3 +n +b$ the \textit{complexity of} $S$. If $S$ is an infinite-type surface, we define its complexity as infinity.\\
%%%%%%%%%%%%%%%%%%%%%%%%%%%%%%%%%%%%%%%%%%%%%%%%
\indent A \textit{curve} on $S$ is a topological embedding $\fun{\alpha}{\mathbb{S}^{1}}{S}$. We say a curve is %\textit{nonperipheral} if it is not isotopic to the boundary curve of a neighbourhood of a puncture, and we say it is 
\textit{essential} if it is not isotopic to the boundary curve of a neighbourhood of a puncture, to a point, nor to a boundary component. All curves are considered to be essential unless otherwise specified.\\
%%%%%%%%%%%%%%%%%%%%%%%%%%%%%%%%%%%%%%%%%%%%%%%%
\indent An \textit{arc} on $S$ is a topological embedding $\fun{\alpha}{I}{S}$, where $I$ is the closed unit interval, and $\alpha(\partial I) \subset \partial S$. We consider all isotopies between arcs to be relative to $\partial I$, \emph{i.e.} the isotopies are not allowed to move the endpoints. We say an arc is \textit{essential} if it is not isotopic to an arc whose image is completely contained in $\partial S$. In this work, all arcs are considered to be essential unless otherwise specified.\\
%%%%%%%%%%%%%%%%%%%%%%%%%%%%%%%%%%%%%%%%%%%%%%%%
\indent We often abuse notation and use the terms ``curve'' and ``arc'' to refer to both the topological embeddings and the corresponding images on $S$. The context makes clear the meaning in each case.\\
%%%%%%%%%%%%%%%%%%%%%%%%%%%%%%%%%%%%%%%%%%%%%%%%
\indent Let $[\alpha]$ and $[\beta]$ be two isotopy classes of essential curves or arcs. The (geometric) \textit{intersection number of} $\alpha$ \textit{and} $\beta$ is defined as follows: $$i([\alpha],[\beta]) \ColonEqq \min \{|(\gamma \cap \delta) \cap \Int{S}| : \gamma \in [\alpha], \delta \in [\beta]\}.$$
%%%%%%%%%%%%%%%%%%%%%%%%%%%%%%%%%%%%%%%%%%%%%%%%
\indent Let $\alpha$ and $\beta$ be two essential curves or arcs on $S$, we say $\alpha$ and $\beta$ are in \textit{minimal position} if $|(\alpha \cap \beta) \cap \Int{S}| = i([\alpha],[\beta])$.

We denote $\rm Homeo^+(S;\partial S)$ the group of orientation-preserving homeomorphisms of $S$ which fix the boundary pointwise.

We conclude this section with a remark that is be fundamental for the rest of the article:

%Remark

\begin{Rem}[Equivalence between homotopy and isotopy]\label{RMK:HomotopyAndIsotopy}
Let $S$ be an infinite-type surface and $f,g\in \rm Homeo(S;\partial S)$. Then $f$ is homotopic to $g$ relative to the boundary if and only if $f$ is isotopic to $g$ relative to the boundary.
\end{Rem}

To simplify the exposition, henceforth we suppose that all homotopies and isotopies in this text are relative to the boundary.

%%%%%%%%%%%%%%%%%%%%%%%%%%%%%%%%%%%%%%%%%%%%%%%%
%%%%%%%%%%%%%%%%%%%%%%%%%%%%%%%%%%%%%%%%%%%%%%%%
%%%%%%%%%%%%%%%%%%%%%%%%%%%%%%%%%%%%%%%%%%%%%%%%
%%%%%%%%%%%%%%%%%%%%%%%%%%%%%%%%%%%%%%%%%%%%%%%%
%\subsection{Homotopy and isotopy}
%	\label{SS:HomoIsot}

%\indent The contrast between homotopy and isotopy can be a delicate point while working on surfaces. Isotopy clearly implies homotopy, but in general the converse is not true.\\
%\indent It is a well-known result (see Proposition 1.10 in \cite{FarbMar}) that homotopy and isotopy are equivalent relations for curves (and arcs) on surfaces of finite topological type. For infinite-type surfaces, this result is proved in the same way.\\
%\indent On the other hand, the same cannot be so easily said about homotopy and isotopy between homeomorphisms. 

For surfaces of finite-topological type the result is well-known, see \cite{Baer}, \cite{Epstein} and Theorem 1.12 in \cite{FarbMar}. For infinite-type surfaces, the results follow from the work of Cantwell and Conlon, see \cite{CCHomoIso}. % \comjhh{Add citation} prove the equivalence if certain conditions on the surface are met (e.g. if $\partial S \neq \varnothing$, then the homotopy must be relative to the boundary).\\
%\indent In this work however, we need not to worry about that since all the surfaces here considered satisfy said conditions, i.e. homotopy relative to the boundary of $S$ is equivalent to isotopy relative to the boundary of $S$.
%%%%%%%%%%%%%%%%%%%%%%%%%%%%%%%%%%%%%%%%%%%%%%%%%
%%%%%%%%%%%%%%%%%%%%%%%%%%%%%%%%%%%%%%%%%%%%%%%%%
%%%%%%%%%%%%%%%%%%%%%%%%%%%%%%%%%%%%%%%%%%%%%%%%%
%%%%%%%%%%%%%%%%%%%%%%%%%%%%%%%%%%%%%%%%%%%%%%%%%
\section{Alexander method}
\label{Sec:AlexanderMethod}

We begin this section by introducing the notion of an \emph{Alexander system} and then some technical lemmas. We then introduce the notion of a \emph{stable} Alexander system. These are collections of arcs and curves such that any $f\in {\rm Homeo^+(S;\partial S)}$ that preserves the isotopy classes of their elements is isotopic to the identity. In the last part of this section we prove theorem \ref{MdeA}, which states that stable Alexander systems exist for infinite type surfaces. We finish the section by showing that for any infinite-type surface $S$ there are actually \emph{uncountably} many different stable Alexander systems in $S$.\\ 
%DEFINITION
\begin{Def}
Let $\Gamma = \{\gamma_{i}\}_{i \in I}$ be a collection of essential curves and arcs on $S$. We say $\Gamma$ is an {\rm Alexander system} if it satisfies the following conditions:
\begin{enumerate}
 \item The elements in $\Gamma$ are in pairwise minimal position.
 \item For $\gamma_{i}, \gamma_{j} \in \Gamma$ with $i \neq j$, we have that $\gamma_{i}$ is not isotopic to $\gamma_{j}$.
 \item For all distinct $i,j,k \in I$, at least one of the following sets is empty: $\gamma_{i} \cap \gamma_{j}$, $\gamma_{j} \cap \gamma_{k}$, $\gamma_{k} \cap \gamma_{i}$.
\end{enumerate}
\end{Def}

\indent Note that any subset of an Alexander system is also an Alexander system and that we do not require the surface $S$ in the preceding definition to be of infinite type. The following result is the infinite-type surface version of Proposition 2.8 in \cite{FarbMar}, to the point that its proof is also completely analogous.
%%%%%%%%%%%%%%%%%%%%%%%%%%%%%%%%%%%%%%%%%%%%%%%%%
%%%%%%%%%%%%%%%%%%%%%%%%%%%%%%%%%%%%%%%%%%%%%%%%%
\begin{Lema}
	\label{MdeAFinito}
 Let $S$ be a connected orientable surface of infinite topological type with possibly nonempty boundary, and $\Gamma=\{\gamma_{1}, \ldots, \gamma_{k}\}$ be a finite Alexander system on $S$. If $h \in \Hom{S,\partial S}$ is such that for all $i = 1, \ldots, k$ we have that $h(\gamma_{i})$ is isotopic to $\gamma_{i}$, then there exists $f \in \Hom{S,\partial S}$ isotopic to the identity on $S$ relative to $\partial S$, such that $f|_{\Gamma} = h|_{\Gamma}$.
\end{Lema}
\begin{Rem}\label{MdeAFNoConexo}
 Note that if $S$ is not connected and $\Gamma$ has only finitely many elements on each connected component of $S$, we can apply either Proposition 2.8 in \cite{FarbMar} or Lemma \ref{MdeAFinito} above to each connected component.
\end{Rem}
Both Lemma \ref{MdeAFinito} and Remark \ref{MdeAFNoConexo} are used repeatedly in the proofs that we present. 

One of the ideas in these proofs is to use convenient families of subsurfaces that exhaust a fixed infinite-type surface. We introduce these below.

%%%%%%%%%%%%%%%%%%%%%%%%%%%%%%%%%%%%%%%%%%%%%%%%%
%%%%%%%%%%%%%%%%%%%%%%%%%%%%%%%%%%%%%%%%%%%%%%%%%
\begin{Def}
Let $\{S_{i}\}$ be an (set-theoretical) increasing sequence of subsurfaces of $S$. We say $\{S_{i}\}$ is a \textit{principal exhaustion of} $S$ if $S = \dsty \bigcup_{i \geq 1} S_{i} $ and for all $i \geq 1$ it satisfies the following conditions:
\begin{enumerate}
 \item $S_{i}$ is a surface of finite topological type,
 \item $S_{i}$ is contained in the interior of $S_{i+1}$,
 \item $\partial S_{i} - \partial S$ is the finite union of pairwise disjoint essential curves on $S$, and
 \item each connected component of $S_{i+1} \backslash S_{i}$ has complexity at least 6.

 %\item For all $i \geq 1$, $$S \backslash \Int{S_{i}} = \bigsqcup_{j = 1}^{n(i)} R_{j}^{i},$$ where each $R_{j}^{i}$ is connected subsurface of $S$ of infinite topological with nonempty boundary. \comjhh{Is this one necessary}
\end{enumerate}
\end{Def} 
\indent Note that these exhaustions always exist in any infinite-type surface. See Figures \ref{LNMExhaustion} and \ref{PulpoGenFinito} for examples.

%%%%%TIKZ PIC
\begin{figure}[ht]
 \begin{center}
  \resizebox{10cm}{!}{\input{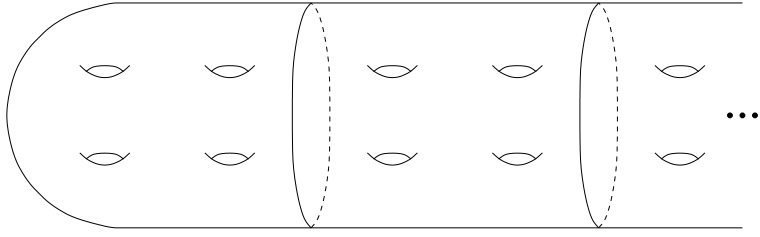}} \caption{An example of a principal exhaustion for the ``Loch Ness Monster'', i.e. the surface of infinite genus and one end.}\label{LNMExhaustion}
 \end{center}
\end{figure}

%%%%%TIKZ PIC
\begin{figure}[ht]
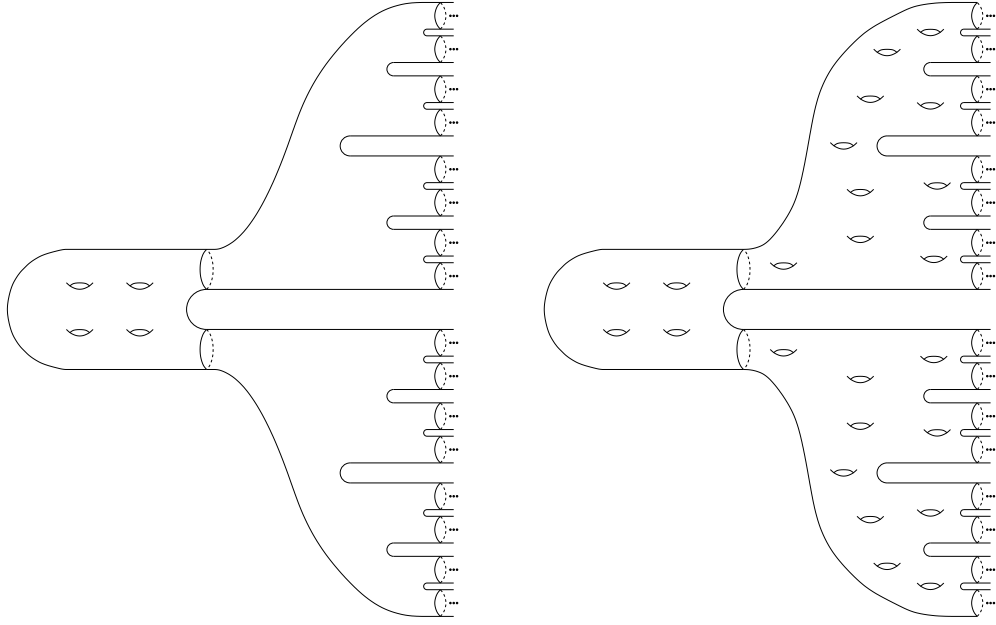

 \begin{center}
  \resizebox{6cm}{!}{\input{PulpoGenFinito.tikz}}\hspace{1cm} \resizebox{6cm}{!}{\input{PulpoGenInfinito.tikz}} \caption{On the left, an example of a principal exhaustion for the surface of genus $4$ minus a Cantor set. On the right, an example of a principal exhaustion for a surface of infinite genus, and whose space of ends that carry genus is a Cantor set.}\label{PulpoGenFinito}
 \end{center}
\end{figure}

\indent Let $\{S_{i}\}$ be a principal exhaustion of $S$. For each $i \geq 1$, we denote by $B_{i}$ the set of the boundary curves of $S_{i}$ that are essential curves on $S$. Note that for $i \neq j$ we have that $B_{i} \cap B_{j} = \varnothing$. We define the \textit{boundaries of} $\{S_{i}\}$, denoted by $B$, as $\dsty B \ColonEqq \bigcup_{i \geq 1} B_{i}$.\\
\indent The first step of the proof of Theorem \ref{MdeA} is the following.
%%%%%%%%%%%%%%%%%%%%%%%%%%%%%%%%%%%%%%%%%%%%%%%%%
%%%%%%%%%%%%%%%%%%%%%%%%%%%%%%%%%%%%%%%%%%%%%%%%%
\begin{Lema}
	\label{FirstStep}
 Let $S$ be an orientable surface of infinite topological type, $\{S_{i}\}$ a principal exhaustion of $S$, and $B$ be the boundaries of $\{S_{i}\}$. If $h \in \Hom{S;\partial S}$ is such that $h(\gamma)$ is isotopic to $\gamma$ for every $\gamma\in B$, then $h$ is isotopic to a homeomorphism $g \in \Hom{S;\partial S}$ for which $g|_{B} = \id{S}|_{B}$.
\end{Lema}
%PROOF
\begin{proof}
 We divide the proof into two parts. In the first part we construct  the homeomorphism $g$ using Lemma \ref{MdeAFinito} and Remark \ref{MdeAFNoConexo}. As a consequence of the construction we have that $g|_{B} = \id{S}|_{B}$. In the second part, using classical results from obstruction theory (see the appendix and references therein), we show that $h$ and $g$ are isotopic.\\
%%%%%%%%%%%%%%%%%%%%%%%%%%%%%%%%%%%%%%%%%%%%%%%%%
\textbf{First part: the construction of $g$.} Remark that for each integer $i \geq 1$, $B_{i}$ is a finite Alexander system, both on $S$ and on $S \backslash S_{k}$ for each $k < i$. Since both $h$ and $B_{1}$ satisfy the conditions of Lemma \ref{MdeAFinito}, there exists $f_{1} \in \Hom{S;\partial S}$ isotopic to $\id{S}$ such that $f|_{B_{1}} = h|_{B_{1}}$. Let $g_{1} \ColonEqq f_{1}^{-1} \circ h$. Then $g_{1}|_{B_{1}} = \id{S}|_{B_{1}}$. As a consequence, we have that $g_{1}(S_{1}) = S_{1}$, and $g_{1}$ also fixes the isotopy classes of $B$. This implies that both the restriction to $S \backslash \Int{S_{1}}$ of $g_{1}$, and $B_{2}$ satisfy the conditions of Remark \ref{MdeAFNoConexo} on $S \backslash \Int{S_{1}}$.\\
 \indent Let $\tld{g}_{1} \ColonEqq g_{1}|_{S \backslash \Int{S_{1}}}$, then by Remark \ref{MdeAFNoConexo}, there exists $\tld{f}_{2} \in \Hom{S \backslash \Int{S_{1}};\partial S}$ isotopic to $\id{S \backslash \Int{S_{1}}}$ relative to $\partial (S \backslash \Int{S_{1}})$ such that $\tld{f}_{2}|_{B_{2}} = \tld{g}_{1}|_{B_{2}}$. Thus, we define the following:
 \begin{center}
  $f_{2}(s) = \left\{ \begin{tabular}{ll}
                       $s$ & if $s \in S_{1}$,\\
                       $\tld{f}_{2}(s)$ & otherwise.
                      \end{tabular}\right.$
 \end{center}
 \indent Note that $f_{2}$ is also isotopic to $\id{S}$. Then, we define $g_{2} \ColonEqq f_{2}^{-1} \circ g_{1} \in \Hom{S;\partial S}$.\\
 \indent By construction, $g_{2}$ satisfies that $g_{2}|_{B_{2}} = \id{S}|_{B_{2}}$, and $g_{2}|_{S_{1}} = g_{1}|_{S_{1}}$. Moreover, $g_{2}$ preserves the connected components of $S \backslash S_{2}$, and $g_{2}$ is isotopic to $g_{1}$.\\
 \indent In this manner, we inductively define for each $n > 2$ a homeomorphism $g_{n} \in \Hom{S;\partial S}$ such that:
 \begin{enumerate}
  \item $g_{n}$ is isotopic to $g_{n-1}$,
  \item $g_{n}|_{B_{n}} = \id{S}|_{B_{n}}$,
  \item for all $m < n$, we have that $g_{n}|_{S_{m}} = g_{m}|_{S_{m}}$.
 \end{enumerate}
 \indent We define the following map:
 \begin{center}
  \begin{tabular}{rccl}
   $g:$ & $S$ & $\rightarrow$ & $S$\\
    & $s$ & $\mapsto$ & $g_{n}(s)$ if $s \in S_{n}$.
  \end{tabular}
 \end{center}
 \indent This map is well-defined since all $g_{n}$ satisfy (3) above. Also, by construction we have that $g \in \Hom{S;\partial S}$ and $g|_{B} = \id{S}|_{B}$.\\
%%%%%%%%%%%%%%%%%%%%%%%%%%%%%%%%%%%%%%%%%%%%%%%%%
\textbf{Second part: the map $g$ is isotopic to $h$.} We claim that  for every $n\geq 1$ the map $g_n$ is homotopic to $g_{n+1}$ relative to $S_n$, that is, there exist a homotopy 
 $\fun{H_{n}}{S \times I}{S}$ satisfying:
 \begin{itemize}
  \item [(a)] $H_{n}|_{S \times \{0\}} = g_{n}$,
  \item [(b)] $H_{n}|_{S \times \{1\}} = g_{n+1}$ and
  \item [(c)] $H_{n}(x,t)=g_n(x)=g_{n+1}(x) \hspace{2mm}\text{for every $t\in I$ and $x\in S_n$}$.
\end{itemize}
We postpone the proof of this claim to the appendix.
%%%%%%%%%%%%%%%%%%%%%%%%%%%%%%%%%%%%%%%%%%%%%%%%%

For each $n\geq 1$, we define the homeomorphism
\begin{center}
  \begin{tabular}{rccl}
   $\zeta_{n}$: & $\left[\frac{n-1}{n},\frac{n}{n+1} \right]$ & $\rightarrow$ & $\left[0,1\right]$\\
    & $t$ & $\mapsto$ & $n\left(n+1\right)\left(t-\frac{n-1}{n}\right)$.
  \end{tabular}
 \end{center}

By conditions (a) and (b) over the family of homotopies $\{H_{n}\}$ the map $H:S\times \left[0,1\right] \to S$ given by 
\begin{center}
  $H(s,t) = \left\{ \begin{tabular}{ll}
                       $H_{n}\left(s,\zeta_{n}\left(t\right)\right)$ & if $t \in \left[\frac{n-1}{n},\frac{n}{n+1} \right]$,\\
                       $g(s)$ & if $t=1$.
                      \end{tabular}\right.$
 \end{center}
is well-defined. We claim that $H$ is a homotopy between $g_1$ and $g$. First, we note that for all $s\in S$, $H(s,0)=H_1(s,\zeta_1(0))=H_1(s,0)=g_1(s)$ and $H(s,1)=g(s)$. It remains to prove that $H$ is continuous. If $(s,t)\in S\times [0,1)$, then by definition, for some $n\geq 1$, $H$ and $H_n$ coincide in some open neighborhood of $(s,t)$. By the continuity of $H_n$, $H$ is continuous at $(s,t)$. Now, let $(s,1) \in S\times \{1\}$. Then there exists $n\geq 1$ such that  $s\in\Int{S_m}$ for all $m\geq n$. Choose $U_s$ an open neighborhood of $s$ properly contained in $\Int{S_m}$.  By condition (c), for all $m\geq n$ and $(s^\prime,t^\prime)\in U_s\times \left(\frac{n-1}{n},1\right]$, we have that $H(s^\prime,t^\prime)=H_m(s^\prime,t^\prime)=g_m(s^\prime)=g(s^\prime)$. Thus, $H$ coincides with $g$ in some open neighborhood of $(s,1)$ and we deduce that $H$ is continuous at $(s,1)$. %In short we have proven that $H$ is continuous. 

Finally, since $h$ is isotopic to $g_1$, we obtain that $h$ is homotopic to $g$. From Remark \ref{RMK:HomotopyAndIsotopy} we conclude that $h$ is isotopic to $g$, as desired. 
\end{proof}
%%%%%%%%%%%%%%%%%%%%%%%%%%%%%%%%%%%%%%%%%%%%%%%%%
\indent Let $N$ be a subsurface of $S$, and $A$ be a collection of curves and arcs on $S$ whose image is contained in $N$. We say $A$ \emph{fills} $N$ if $\Int{N \backslash A}$ is the disjoint union of open discs and once-punctured open discs. %This definition is independent of the topological type of $S$.
%DEF. STABLE ALEXANDER SYSTEM
\begin{Def}
Let $\Gamma$ be an Alexander system on $S$. We say $\Gamma$ is a \textit{stable Alexander system} if $\Gamma$ fills $S$ and every $f\in \Hom{S;\partial S}$ that preserves the isotopy classes of elements in $\Gamma$ is isotopic relative to the boundary, to the identity. 
\end{Def}

\begin{Rem}
	\label{RemStableAS}
In the preceding definition we do not require $S$ to be of infinite type. Also note that not every Alexander system that fills $S$ is stable. See Figures \ref{CEStabASyst} and \ref{PuncturedSurface} for a counterexamples and an example of stable Alexander systems on finite-type surfaces respectively.
\end{Rem}

%%%%%TIKZ PIC
\begin{figure}[ht]
 \begin{center}
  \resizebox{10cm}{!}{\input{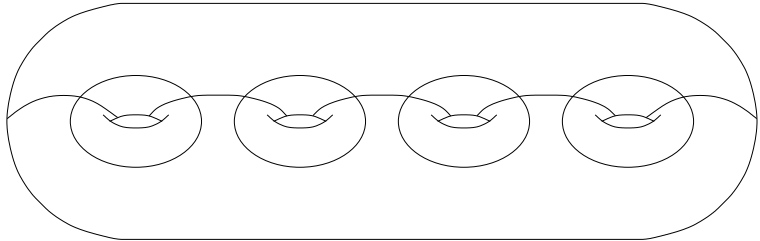}} \caption{An Alexander system that fills $S$ but is not stable.}\label{CEStabASyst}
 \end{center}
\end{figure}

\begin{figure}[ht]
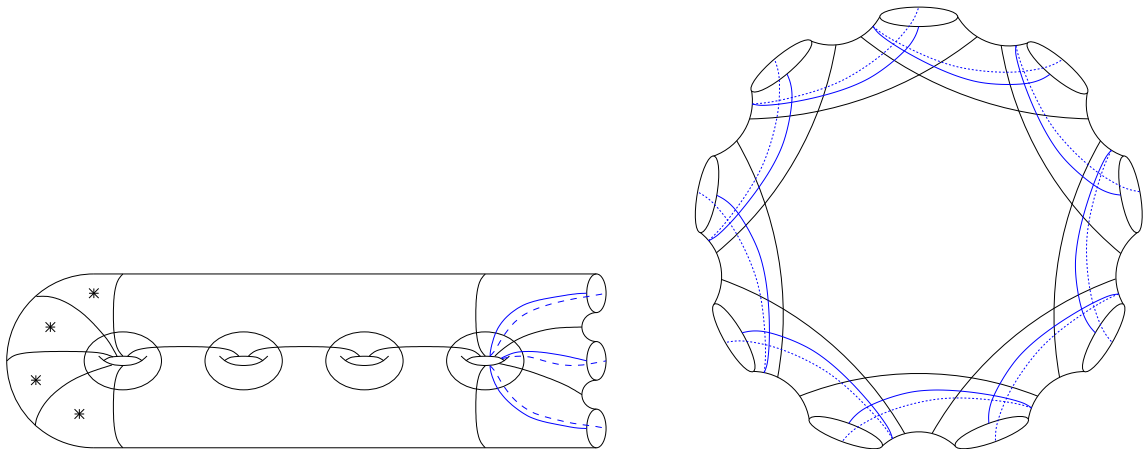

 \begin{center}
  \resizebox{8cm}{!}{\input{PuncturedSurface.tikz}}\hspace{1cm} \resizebox{6cm}{!}{\input{SphereNinegon.tikz}} \caption{On the left, an example of a stable Alexander system for a surface of genus $4$, with $4$ punctures and $3$ boundary components. On the right, an example of a stable Alexander system for a surface of genus $0$ with $9$ boundary components.}\label{PuncturedSurface}
 \end{center}
\end{figure}

%%%%%%%%%%%%%%%%%%%%%%%%%%%%%%%%%%%%%%%%%%%%%%%%%
\begin{proof}[\textbf{Proof of Theorem \ref{MdeA}}] First we construct an Alexander system $\Gamma$ that fills $S$ and then we show that this system is stable. 
 
Let $\{S_n\}$ be a principal exhaustion of $S$, and $B$ be the boundaries of $\{S_n\}$. For each $n\geq 1$, choose a finite stable Alexander system $C_n$ for $\Int{S_{n} \backslash S_{n-1}}$ (with $S_0=\varnothing$) and let $C$ be the union of the collection of $\{C_n\}$. See figure 
\ref{PuncturedSurface} for examples of stable Alexander systems.

Note that $B \cup C$ is not a stable Alexander system. There are homeomorphism which are not homotopic to the identity but that fix all the (isotopy classes of) elements in this collection, \emph{e.g.} any representative of a Dehn-twists along a curve in $B$. As we prove below, these are the only homeomorphisms showing this behaviour. 

Let $B^*$ be a collection of curves, and $\fun{f}{B}{B^{*}}$ a bijection satisfying that for all $\gamma \in B$ and all $\delta \in B \backslash \{\gamma\}$, $i([\gamma], [f(\gamma)]) \neq 0$ and $i([\delta], [f(\gamma)]) = 0$.
Remark that we can choose $B^*$ so that $\Gamma:=B\cup B^{*}\cup C$ is an Alexander system. 
Figure \ref{GammaPulpoGenFin} illustrates how to choose the collection $B^*$. We finish this proof by showing that $\Gamma$ is a \emph{stable} Alexander system. 

%%%%%TIKZ PIC
\begin{figure}[ht]
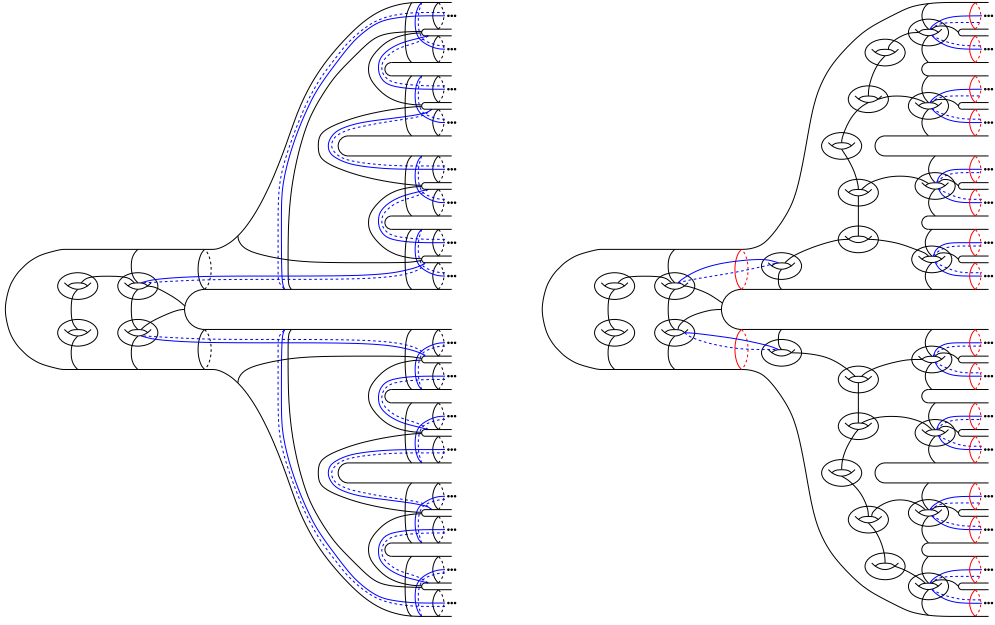

 \begin{center}
  \resizebox{6cm}{!}{\input{GammaPulpoGenFin.tikz}}\hspace{1cm} \resizebox{6cm}{!}{\input{GammaPulpoGenInf.tikz}} \caption{On the left, an example of $\Gamma$ for the surface of genus $4$ minus a Cantor set. On the right, an example of $\Gamma$ for a surface of infinite genus, and whose space of ends that carry genus is a Cantor set.}\label{GammaPulpoGenFin}
 \end{center}
\end{figure}

% \begin{itemize}
  %\item $B \cap B^{*} = C \cap B^{*} = \varnothing$.
 % \item There exists a bijection $\fun{f}{B}{B^{*}}$ such that for all $\gamma \in B$ and all $\delta \in B \backslash \{\gamma\}$, $i(\gamma, f(\gamma)) \neq 0$ and $i(\delta, f(\gamma)) = 0$.
 % \item The set $C_{n} \cup (B^{*} \cap (S_{n} \backslash \Int{S_{n-1}}))$ is a stable Alexander system for $S_{n} \backslash \Int{S_{n-1}})$.
 %\end{itemize}
 
%We define $\Gamma:=B\cup B^{*}\cup C$. By construction, $\Gamma$ is an Alexander system of curves on $S$ that fills $S$. We must prove now that $\Gamma$ is stable.\\

Let $h\in \Hom{S;\partial S}$ be such that $h$ fixes the isotopy class of every curve in $\Gamma$. We claim that $h$ is isotopic to the identity on $S$. 

First, since $h$ fixes the isotopy class of elements in $B\subset \Gamma$, by the Lemma \ref{FirstStep}, $h$ is isotopic to some $g\in \Hom{S;\partial S}$ such that $g|_{B} = \id{S}|_{B}$. In particular, $g$ also fixes all the isotopy classes of elements in $\Gamma$. 
 %\indent \comjhh{This paragraph follows from h being isotopic to g, doesn't it?}Now, let $n\geq 1$. We will prove that $g$ fix the isotopy class of arcs and curves of $\left(S_{n} \backslash \Int{S_{n-1}}\right)\cap \Gamma$. We note that $\left(S_{n} \backslash \Int{S_{n-1}}\right)\cap \Gamma=C_n\cup D_n$, where $D_n$ is the union of arcs in $\left(S_{n} \backslash \Int{S_{n-1}}\right)\cap B^{*}$. By hipothesis, we have that $g$ fix the isotopy class of curves in $C_n$. We claim that also $g$ fix the isotopy class of arcs in $D_n$. In effect, on the contrary, exist $\gamma^{*}\in B^{*}$ such that it is not isotopic to $g(\gamma^{*})$, which contradicts that $g$ preserve the isotopy class of curves in $\Gamma$.\\

Observe that for $n\geq 1$,  $S_{n} \backslash \Int{S_{n-1}}$ is a disjoint union of a finite number of finite-type surfaces with boundary, and by construction $\left(S_{n} \backslash \Int{S_{n-1}}\right)\cap \Gamma$ is a finite stable Alexander system on $S_{n} \backslash \Int{S_{n-1}}$. Then by definition of stable Alexander systems, the restriction of $g$ to $S_{n} \backslash \Int{S_{n-1}}$ is isotopic to the identity on $S_{n} \backslash \Int{S_{n-1}}$ relative to the boundary $\partial (S_{n} \backslash \Int{S_{n-1}})$. Given that $g|_{B} = \id{S}|_{B}$, the homeomorphism $g$ is isotopic to the identity on $S$.
\end{proof} 
%%%%%%%%%%%%%%%%%%%%%%%%%%%%%%%%%%%%%%%%%%%%%%%%%
%%%%%%%%%%%%%%%%%%%%%%%%%%%%%%%%%%%%%%%%%%%%%%%%%
\begin{Lema}\label{UncountablyMany}
 Let $S$ be an orientable connected surface of infinite topological. Then there exists uncountably many isotopy classes of stable Alexander systems on $S$.
\end{Lema}
\begin{proof}
 Fix $\{S_{i}\}$ a principal exhaustion of $S$, and let $\Gamma = B \cup C \cup B^{*}$ be as in the proof of Theorem \ref{MdeA}. For each $i \geq 1$ let $\varphi_{i} \in \Hom{\Int{S_{i} \backslash S_{i-1}};\partial S}$ (with $S_{0} = \varnothing$) be a map that preserves each connected component of $\Int{S_{i} \backslash S_{i-1}}$ and whose restriction to each connected component is a pseudo-Anosov map.
 
Let $\mathcal{A}$ denote the set of all isotopy classes of stable Alexander systems of $S$. Note that $\mathcal{A}$ is invariant under the action of $\Hom{S;\partial S}$; hence we can define the following map:
 \begin{center}
  \begin{tabular}{rlcl}
   $\Phi :$ & $\{0,1\}^{\Zedp}$ & $\rightarrow$ & $\mathcal{A}$\\
    & $(\epsilon_{i})_{i \geq 1}$ & $\mapsto$ & $\dsty \left[B \cup \left(\bigcup_{i \geq 1} \varphi_{i}^{\epsilon_{i}}(C_{i})\right) \cup \left(\bigcup_{i \geq 1} \varphi_{i}^{\epsilon_{i}}(f(B_{i}))\right)\right]$;
  \end{tabular}
 \end{center}
 where $f$ is the bijection between $B$ and $B^{*}$. Since each $\varphi_{i}$ is a pseudo-Anosov map, the map $\Phi$ is injective. Thus, $\mathcal{A}$ is an infinitely uncountable set.
\end{proof}
%\commf{The proof of the preceding lemma is correct but I would like to revisit the details to see if we can explain some more to the reader.}
%%%%%%%%%%%%%%%%%%%%%%%%%%%%%%%%%%%%%%%%%%%%%%%%%
%%%%%%%%%%%%%%%%%%%%%%%%%%%%%%%%%%%%%%%%%%%%%%%%%
%%%%%%%%%%%%%%%%%%%%%%%%%%%%%%%%%%%%%%%%%%%%%%%%%
%%%%%%%%%%%%%%%%%%%%%%%%%%%%%%%%%%%%%%%%%%%%%%%%%
\appendix
\section{Appendix.}
\label{Appendix}
In this appendix we explain how classical results of obstruction theory are used to assure the existence of the homotopies $H_{n}$ used in the proof of lemma \ref{FirstStep}. Our discussion is based on the work of P. Olum \cite{Olum} and adapted to the context of orientable surfaces with boundary. 

\emph{The extension problem}. 
Let $S$ be an orientable surface\footnote{$S$ is not necessarily compact, nor necessarily of finite-type and its boundary might be non-empty.}  and $S'\subset S$ a subsurface. Consider two continuous functions $f_0,f_1:S\to S$ such that $f_0(s)=f_1(s)$ for all $s\in S'$. Following \cite{Olum}, let $\overline{S}_{01}'$ be the subset of $S\times I$ formed by $S\times{0}\cup S'\times I\cup S\times\{1\}$. Define $F:\overline{S}_{01}'\to S$ by:
\begin{equation}
F(s,t)=\begin{cases}
        f_0(s)\hspace{3mm}\text{for $(s,t)\in S\times\{0\}\cup S'\times I$}\\
        f_1(s)\hspace{3mm}\text{for $(s,t)\in S\times\{1\}$}        \end{cases}
\end{equation}
We say that  \emph{$f_0$ is homotopic to $f_1$ relative to $S'$} if $F$ defined above has a continuos extension to $S\times I$. The following result is an adaptation of theorem 25.2 in [\emph{Ibid.}] to the context of orientable surfaces with boundary; it gives criteria to determine when $f_0$ and $f_1$ are homotopic relative to $S'$. We state the result first, then explain its content and finally show how to apply to the proof of lemma \ref{FirstStep}.
%OLUM Theorem
\begin{Teo}
	\label{THM:Olum}
Let $f_0,f_1:S\to S$ be two continuous functions such that $f_0(s)=f_1(s)$ for all $s\in S'$. Let $s_0\in S'$ and $\theta_i:\pi_1(S,s_0)\to\pi_1(S,f_i(s_0))$ be the homomorphism induced by $f_i$, for $i=1,2$. For $k\in\{0,1,2\}$ fixed the following statements are equivalent:
\begin{enumerate}
\item (For $k=2$) $\mathbf{O}^k(f_0,f_1)\hspace{1mm}\rm rel.\hspace{1mm}S'$ is non-void and contains the zero element.
\item (For $1\leq k$) $f_0\simeq f_1\hspace{1mm}\text{\rm dim $k$ (rel $S'$)}$ .
\item (For $1\leq k$) $\mathbf{O}^{k+1}(f_0,f_1)\hspace{1mm}\rm rel.\hspace{1mm}S'$ is non-void.
\item The homomorphism $\theta_0$ and $\theta_1$ induced by $f_0$ and $f_1$ are equal.
\end{enumerate}
\end{Teo}

We now explain the content of this theorem. Let $\tau$ be a triangulation of $S$ and $\tau^k$ its $k$-skeleton (in particular $\tau^2=S$). Two functions $f_0$ and $f_1$ satisfying the hypothesis of the preceding theorem are \emph{homotopic in dimension $k$ relative to $S'$}, written $f_0\simeq f_1\hspace{1mm}\text{\rm dim $k$ (rel $S'$)}$, if $f_0$ restricted to $S'\cup \tau^k$ is homotopic to $f_1$ restricted to 
$S'\cup \tau^k$ relative to $S'$. This is precisely the content of point (2) above. In points (1) and (3) appears $\mathbf{O}^k(f_0,f_1)\hspace{1mm}\rm rel.\hspace{1mm}S'$, the \emph{$k^{th}$ obstruction to a homotopy of $f_0$ to $f_1$ relative to $S'$}. This is a subset of the cohomology group $H^k(S,S',\theta^*_0\pi_n)$. Here $\theta^*_0\pi_k$ is the system of (twisted) local groups $\theta^*_0\pi_k(S,s_0)$, where $\pi_k(S,s_0)$ is the $k^{th}$-homotopy group of $S$ with based at $s_0$. The definition of both the obstruction and the cohomology groups is well explained in [\emph{Ibid.}], but as we will see we do not need them because in our context all these objects are trivial. 
 
\emph{Constructing the homotopies $H_{n}$}. Recall from the proof of lemma \ref{FirstStep} that for each $n\in\mathbb{N}$ we have two isotopic maps $g_n,g_{n+1}\in {\rm Homeo}^+(S;\partial S)$ which coincide in $S_n\subset S$. Since the maps are isotopic, the isomorphisms that they induce in 
a fundamental group of $S$ with basepoint in $S_n$ are equal. In other words, these maps satisfy (4) in theorem \ref{THM:Olum}. 

By (3) above we have that 
$\mathbf{O}^2(f_0,f_1)\hspace{1mm}\rm rel.\hspace{1mm}S'$ is non-void. 
Since $S$ is not a sphere, the cohomology groups $H^2(S,S_n,\theta_0^*\pi_n)$ are trivial because all local coefficients $\pi_2(S,s_0)$ are trivial. Therefore, $\mathbf{O}^2(f_0,f_1)$ only contains the zero element, satisfying thus (1) for $k=2$. In particular, by (2) above we can conclude that 
$g_n$ and $g_{n+1}$ are homotopic in dimension 2 relative to $S'$, which is precisely what we wanted.

%BIBLIOGRAPHY  

\end{document}